\documentclass{scrartcl}
\usepackage[T1]{fontenc}
\usepackage[utf8]{inputenc}
\usepackage[english]{babel}
\usepackage{amsmath,amsthm,amsfonts,dsfont,mathtools}
\usepackage{graphicx}
\usepackage{prettyref}
\newrefformat{def}{Definition \ref{#1}}
\usepackage[format=hang]{caption}
\usepackage[labelformat=simple]{subcaption}

\usepackage[unicode=true,pdfusetitle,%
 bookmarks=true,bookmarksnumbered=false,bookmarksopen=false,%
 breaklinks=false,pdfborder={0 0 1},backref=false,colorlinks=false]
 {hyperref}
\hypersetup{%
   pdfauthor={Daniel Karrasch, Florian Huhn, George Haller},%
   pdftitle=Automated detection of coherent Lagrangian vortices in two-dimensional unsteady flows,%
   pdfkeywords= {Coherent Lagrangian vortices, Transport, Index theory, Line fields, Closed orbit detection, Ocean surface flows}
}

\theoremstyle{definition}
\newtheorem{defn}{Definition}
\theoremstyle{plain}
\newtheorem{theorem}{Theorem}
\newtheorem{proposition}{Proposition}

\newrefformat{prop}{Proposition \ref{#1}}

\DeclareMathOperator{\ind}{ind}

\begin{document}

\title{Automated detection of coherent Lagrangian vortices in two-dimensional unsteady flows}

\author{Daniel Karrasch
\and Florian Huhn
\and George Haller \and
Institute of Mechanical Systems\\ ETH Zürich, Leonhardstrasse 21\\ 8092 Zürich, Switzerland
}

\maketitle

\begin{abstract}
Coherent boundaries of Lagrangian vortices in fluid flows have recently
been identified as closed orbits of line fields associated
with the Cauchy--Green strain tensor. Here we develop a fully automated
procedure for the detection of such closed orbits in large-scale velocity
data sets. We illustrate the power of our method on ocean surface
velocities derived from satellite altimetry.
\end{abstract}


\section{Introduction}

Lagrangian coherent structures (LCS) are exceptional material surfaces
that act as cores of observed tracer patterns in fluid flows (see
\cite{Peacock2010} and \cite{Peacock2013} for reviews). For oceanic
flows, the tracers of interest include salinity, temperature, contaminants,
nutrients and plankton---quan\-ti\-ties that play an important role in
the ecosystem and even in climate. Fluxes of these quantities
are typically dominated by advective transport over diffusion.

An important component of advective transport in the ocean is governed
by mesoscale eddies, i.e., vortices of $100$--$200$\,km in diameter. While eddies also stir and mix surrounding water masses by their swirling motion, here we focus on eddies that trap and carry fluid in a coherent manner. Eddies of this kind include the Agulhas rings
of the Southern Ocean. They are known to transport massive quantities
of warm and salty water from the Indian Ocean into the Atlantic Ocean
\cite{Ruijter1999}.
Current limitations on computational power prevent that climate models resolve mesoscale eddies in their flow field. Since the effect of mesoscale eddies on the global circulation is significant \cite{Wolfe2009}, the correct parameterization of eddy transport is crucial for the reliability of these models. As a consequence, there is a rising interest in systematic and accurate eddy detection and census in large global data sets, as well as in quantifying the average transport of trapped fluid by all eddies in a given region \cite{Dong2014, Zhang2014, Petersen2013}.

This quantification requires (i) a rigorous method that provides specific coherent eddy boundaries, and (ii) a robust numerical implementation of the method on large velocity data sets.

A number of vortex definitions have been proposed in the literature, \cite{Haller2005,Zhang2014a},
most of which are of Eulerian type, i.e., use information from the instantaneous velocity field. Typical global eddy studies \cite{Chelton2007, Dong2014, Zhang2014, Petersen2013} are based on such Eulerian approaches. Evolving eddy boundaries obtained from Eulerian approaches, however, do not encircle and transport the same body of water coherently \cite{Haller2013a,Zhang2014a}. Instead, fluid initialized within an instantaneous Eulerian eddy boundary will generally stretch, fold and filament significantly. Yet only coherently transported scalars resist erosion by diffusion in a way that a sharp
signature in the tracer field is maintained. All this suggests that coherent eddy transport should ideally be analysed via Lagrangian methods that take into account the evolution of trajectories in the flow, such as, e.g., \cite{Provenzale1999,Allshouse2012,Mendoza2010, Rypina2011,Tallapragada2013,Prants2013}. Notably, however, none of these methods focuses on the detection of vortices and none provides an algorithm to extract exact eddy boundaries in unsteady velocity fields.

Only recently have mathematical approaches emerged for the detection
of coherent Lagrangian vortices. These include the geometric approach
\cite{Haller2005,Haller2013a} and the set-oriented approach \cite{Froyland2010a,Froyland2013,Froyland2012a}. Here, we follow
the geometric approach to coherent Lagrangian vortices, which defines
a coherent material vortex boundary as a closed stationary curve of
the averaged material strain \cite{Haller2013a}. All solutions
of this variational problem turn out to be closed material curves
that stretch uniformly. Such curves are practically found as closed orbits of appropriate planar line
fields \cite{Haller2013a}.

In contrast to vector fields, line fields are special vector bundles over
the plane. In their definition, only a one-dimensional subspace (line)
is specified at each point, as opposed to a vector at each point.
The importance of line field singularities in Lagrangian eddy detection
has been recognized in \cite{Haller2013a}, but has remained only
partially exploited. Here, we point out a topological rule that enables
a fully automated detection of coherent Lagrangian vortex boundaries
based on line field singularities. This in turn makes automated Lagrangian eddy detection feasible for large ocean regions.

Based on the geometric approach, coherent Lagrangian vortices have
so far been identified in oceanic data sets \cite{Beron-Vera2013,Haller2013a},
in a direct numerical simulation of the two-dimensional Navier--Stokes equations
\cite{Farazmand2014}, in a smooth area-preserving map \cite{Haller2012},
in a kinematic model of an oceanic jet in \cite{Haller2012}, and
in a model of a double gyre flow \cite{Onu2014}. With the exception
of \cite{Haller2013a}, however, these studies did not utilize the
topology of line field singularities. Furthermore, none of them offered
an automated procedure for Lagrangian vortex detection.

The orbit structure of line fields has already received considerable
attention in the scientific visualization community (see \cite{Delmarcelle1994a,Weickert2006} for reviews). The problem of
closed orbit detection has been posed in \cite[Section 5.2.3]{Delmarcelle1994a},
and was considered by \cite{Wischgoll2006},
building on \cite{Wischgoll2001}. In that approach,
numerical line field integration is used to identify cell chains that
may contain a closed orbit. Then, the conditions of the Poincar\'e--Bendixson
theorem are verified to conclude the existence of a closed orbit for
the line field. This approach, however, does not offer a systematic
way to search for closed orbits in large data sets arising in geophysical
applications.

This paper is organized as follows. In \prettyref{sec:Index_vf},
we recall the index theory of planar vector fields. In \prettyref{sec:Index_lf},
we review available results on indices for planar line fields, and
deduce a topological rule for generic singularities inside closed
orbits of such fields. Next, in \prettyref{sec:Application}, we present
an algorithm for the automated detection of closed line field orbits.
We then discuss related numerical results on ocean data, before presenting
our concluding remarks in \prettyref{sec:Conclusions}.

\section{Index theory for planar vector fields\label{sec:Index_vf}}

Here, we recall the definition and properties of the index
of a planar vector field \cite{Needham2000}. We denote the unit circle of the plane by $\mathcal{S}^{1}$, parametrized by the mapping $(\cos2\pi s,\sin2\pi s)\in\mathcal{S}^{1}\subset\mathds{R}^{2}$,
$s\in\left[0,1\right]$. In our notation, we do not distinguish between
a curve $\gamma\colon[a,b]\to\mathds{R}^{2}$ as a function and its image
as a subset of $\mathds{R}^{2}$.

\begin{defn}[Index of a vector field]
\label{def:index_vf} For a continuous, piecewise differentiable
planar vector field $\mathbf{v}\colon D\subseteq\mathds{R}^{2}\to\mathds{R}^{2}$
and a simple closed curve $\gamma\colon\mathcal{S}^{1}\to\mathds{R}^{2}$, let $\theta\colon\left[0,1\right]\to\mathds{R}$
be a continuous function such that $\theta(s)$ is the angle between
the $x$-axis and $\mathbf{v}(\gamma(s))$. Then, the \emph{index}
(or \emph{winding number}) \emph{of~ $\mathbf{v}$ along $\gamma$}
is defined as
\[
\ind_{\gamma}(\mathbf{v})\coloneqq\frac{1}{2\pi}\left(\theta(1)-\theta(0)\right),
\]
i.e., the number of turns of $\mathbf{v}$ during one anticlockwise
revolution along $\gamma$. Clearly, $\theta$ is well-defined only
if there is no \emph{critical point} of $\mathbf{v}$ along $\gamma$,
i.e., no point at which $\mathbf{v}$ vanishes.
\end{defn}

The index defined in \prettyref{def:index_vf} has two important properties \cite{Perko2001}:
\begin{enumerate}
\item \emph{Decomposition property}:
\[
\ind_{\gamma}(\mathbf{v})=\ind_{\gamma_{1}}\left(\mathbf{v}\right)+\ind_{\gamma_{2}}\left(\mathbf{v}\right),
\]
whenever $\gamma=\gamma_{1}\cup\gamma_{2}\setminus(\gamma_{1}\cap\gamma_{2})$,
and $\ind_{\gamma_{i}}\left(\mathbf{v}\right)$ are well-defined.
\item \emph{Homotopy invariance}:
\[
\ind_{\gamma}\left(\mathbf{v}\right)=\ind_{\tilde{\gamma}}\left(\mathbf{v}\right),
\]
whenever $\tilde{\gamma}$ can be obtained from $\gamma$ by a continuous
deformation (homotopy).
\end{enumerate}
If $\gamma$ encloses exactly one critical point $p$ of $\mathbf{v}$, then the \emph{index of $p$ with respect to $\mathbf{v}$},
\[
\ind\left(p,\mathbf{v}\right)\coloneqq\ind_{\gamma}\left(\mathbf{v}\right)
\]
is well-defined, because its definition does not depend on the particular
choice of the enclosing curve by homotopy invariance. Furthermore,
the index of $\gamma$ equals the sum over the indices of all enclosed
critical points, i.e.,
\[
\ind_{\gamma}\left(\mathbf{v}\right)=\sum_{i}\ind(p_{i},\mathbf{v}),
\]
provided all $p_{i}$ are isolated critical points. Finally, the index
of a closed orbit $\Gamma$ of the vector field $\mathbf{v}$ is equal
to $1$, because the vector field turns once along $\Gamma$. Therefore,
closed orbits of planar vector fields necessarily enclose critical
points.

\section{Index theory for planar line fields\label{sec:Index_lf}}

We now recall an extension of index theory from vector fields to line
fields \cite{Spivak1999}. Let $\mathbb{P}^{1}$ be the set of one-dimensional subspaces of $\mathds{R}^{2}$,
i.e., the set of lines through the origin $0\in\mathds{R}^{2}$. $\mathbb{P}^{1}$
is sometimes also called the \emph{projective line}, which can be
endowed with the structure of a one-dimensional smooth manifold \cite{Lee2012}.
This is achieved by parametrizing the lines via the $x$-coordinate
at which they intersect the horizontal line $y=1$. The horizontal
line $y=0$ is assigned the value $\infty$.

Equivalently, elements
of $\mathbb{P}^{1}$ can be parametrized by their intersection with
the upper semi-circle, denoted $\mathcal{S}_{+}^{1}$, with its right and left
endpoints identified. This means that lines through the origin are
represented by a unique normalized vector, pointing in the upper half-plane
and parametrized by the angle between the representative vector and
the $x$-axis (Fig.\ \ref{fig:circle-semi-circle}). A \emph{planar
line field} is then defined as a mapping $\mathbf{l}\colon D\subseteq\mathds{R}^{2}\to\mathbb{P}^{1}$,
with its differentiability defined with the help of the manifold structure
of $\mathbb{P}^{1}$.

\begin{figure}
\centering
\includegraphics[width=0.45\textwidth]{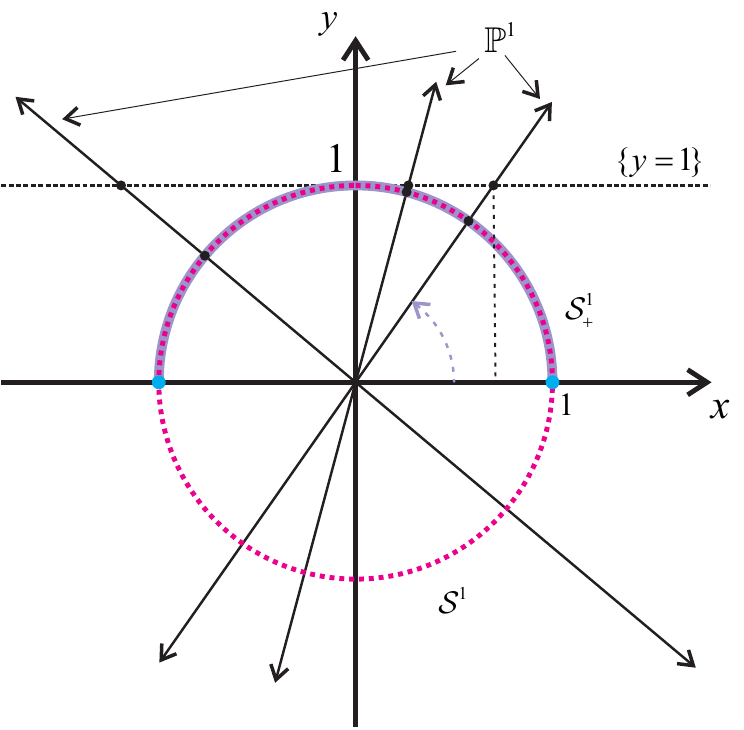}
\caption{The geometry of the projective line and its parametrization. The double-headed
arrows represent one-di\-men\-sion\-al subspaces of the plane, i.e., elements
of $\mathbb{P}^{1}$. The upper semi-circle $\mathcal{S}_{+}^{1}$ is shown
in purple, its end-points in cyan, and the unit circle $\mathcal{S}^1$ in dashed magenta.
The black points represent intersections of the lines with $y=1$
and with the unit circle, respectively.}
\label{fig:circle-semi-circle}
\end{figure}

Line fields arise in the computation of eigenvector fields for symmetric,
second-order tensor fields \cite{Delmarcelle1994,Tricoche2006}.
Eigenvectors have no intrinsic sign or length: only eigenspaces are
well-defined at each point of the plane. Their orientation depends
smoothly on their base point if the tensor field is smooth and has
simple eigenvalues at that point. At repeated eigenvalues, isolated
one-dimensional eigenspaces (and hence the corresponding values of the line field)
become undefined.

Points to which a line field cannot be extended continuously are called
\emph{singularities}. These points are analogous to critical points
of vector fields. Away from singularities, any smooth line field can
locally be endowed with a smooth orientation. This implies the local
existence of a normalized smooth vector field, which pointwise spans
the respective line. Conversely, away from critical
points, smooth vector fields induce smooth line fields when one takes their
linear span pointwise.

Based on the index for planar vector fields, we introduce a notion
of index for planar line fields following \cite{Spivak1999}. First,
for some differentiable line field $\mathbf{l}$ and along some closed
curve $\gamma\colon\mathcal{S}^{1}\to\mathds{R}^{2}$, pick at each point $\gamma(t)$
the representative upper half-plane vector from $\mathbf{l}(\gamma(t))$.
This choice yields a normalized vector field along $\gamma$ which
is as smooth as $\mathbf{l}$, except where $\mathbf{l}\circ\gamma$
crosses the horizontal subspace. At such a point, there is a jump-discontinuity
in the representative vector from right to left or vice versa. To
remove this discontinuity, the representative vectors are turned counter-clockwise
by $\alpha\colon\mathcal{S}_{+}^{1}\to\mathcal{S}^{1}$, $\left(\cos2\pi s\sin2\pi s\right)\mapsto\left(\cos4\pi s,\sin4\pi s\right)$,
$s\in[0,1/2]$, i.e., the parametrizing angle is doubled. Thereby,
the left end-point with angle $\pi$ is mapped onto the right end-point
with angle $0$. This representation $\alpha\circ\mathbf{l}$ permits
the extension of the notion of index to planar line fields as follows.

\begin{defn}[Index of a line field]\label{def:index_lf}
For a continuous, piecewise differentiable
planar line field $\mathbf{l}\colon D\subseteq\mathds{R}^{2}\to\mathbb{P}^{1}$
and a simple closed curve $\gamma\colon\mathcal{S}^{1}\to\mathds{R}^{2}$, we define
the \emph{index of~ $\mathbf{l}$ along $\gamma$} as
\[
\ind_{\gamma}(\mathbf{l})\coloneqq\frac{1}{2}\ind_{\gamma}(\alpha\circ\mathbf{l}).
\]
\end{defn}

The coefficient $1/2$ in this definition is needed to correct
the doubling effect of $\alpha$. It also makes the index for a line
field, generated by a vector field in the interior of $\gamma$, equal
to the index of that vector field. Since \prettyref{def:index_lf} refers to
\prettyref{def:index_vf}, the additional definitions and properties described
in \prettyref{sec:Index_vf} for vector fields carry over to line
fields.

We call a curve $\gamma$ an \emph{orbit} of $\mathbf{l}$, if it
is everywhere tangent to $\mathbf{l}$. The scientific visualization
community refers to orbits of lines fields arising from the eigenvectors
of a symmetric tensor as \emph{tensor (field) lines} or \emph{hyperorbit
(trajectories}) \cite{Delmarcelle1994,Tricoche2006,Wischgoll2006}.

By definition, the index of singularities of line fields can be a
half integer, as opposed to the vector field case, where only integer
indices are possible. Also, two new types of singularities emerge
in the line field case: \emph{wedges} (type $W$) of index $+1/2$,
and \emph{trisectors} (type $T$) of index $-1/2$ \cite{Delmarcelle1994,Tricoche2006}.
The geometry near these singularities is shown in Fig.\ \ref{fig:deg_points}.

\begin{figure}
\centering
\includegraphics[width=0.5\textwidth]{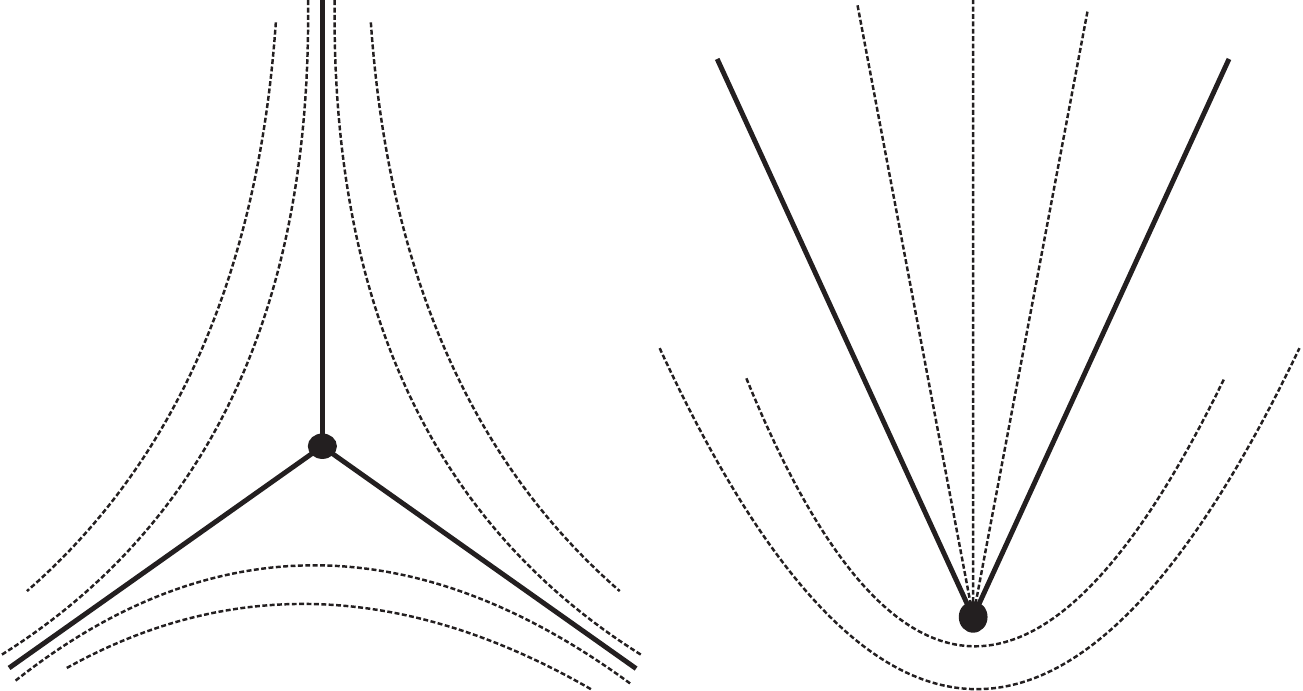}
\caption{Orbit topologies in the vicinity of the two generic line field
singularity types: trisector (left) and wedge (right). All lines represent
orbits, the solid lines correspond to boundaries of hyperbolic sectors.}
\label{fig:deg_points}
\end{figure}

Node, centre, focus and saddle singularities also exist for line fields,
but these singularities turn out to be structurally unstable with
respect to small perturbations to the line field \cite{Delmarcelle1994}.

In this paper, we assume that only \emph{isolated} singularities of
the \emph{generic} wedge and trisector types are present in the line
field of interest. In that case, we obtain the following topological
constraint on closed orbits of the line field.

\begin{theorem}\label{thm:topo_line}
Let $\mathbf{l}$ be a continuous, piecewise
differentiable line field with only structurally stable singularities.
Let $\Gamma$ be a closed orbit of~ $\mathbf{l}$, and let $D$ denote
the interior of $\Gamma$. We then have
\begin{equation}
W=T+2,\label{eq:toporule_line}
\end{equation}
where $W$ and $T$ denote the number of wedges and trisectors, respectively,
in $D$.
\end{theorem}

\begin{proof}
First, $\Gamma$ has index $1$ with respect to $\mathbf{l}$, i.e.,
$\ind_{\Gamma}\left(\mathbf{l}\right)=1$. Second, its index equals
the sum over all enclosed singularities, i.e.,
\begin{equation}
\sum_{i}\ind_{\Gamma}(p_{i},\mathbf{l})=\ind_{\Gamma}(\mathbf{l})=1.\label{eq:sum}
\end{equation}
Since we consider structurally stable singularities only, these are
isolated and of either wedge or trisector type. From \eqref{eq:sum},
we then obtain the equality
\[
\frac{1}{2}\left(W-T\right)=1,
\]
from which Eq.\ \eqref{eq:toporule_line} follows.
\end{proof}

Consequently, in the interior of any closed orbit of a structurally stable line field,
there are at least two singularities of wedge type, and exactly two
more wedges than trisectors. Thus, a closed orbit necessarily encircles
a wedge pair, and hence the existence of such a pair serves as a necessary
condition in an automated search for closed orbits in line fields.
In Fig.\ \ref{fig:topo_cycle}, we sketch two possible line field
geometries in the interior of a closed orbit.

\begin{figure}
\centering
\includegraphics[width=0.5\textwidth]{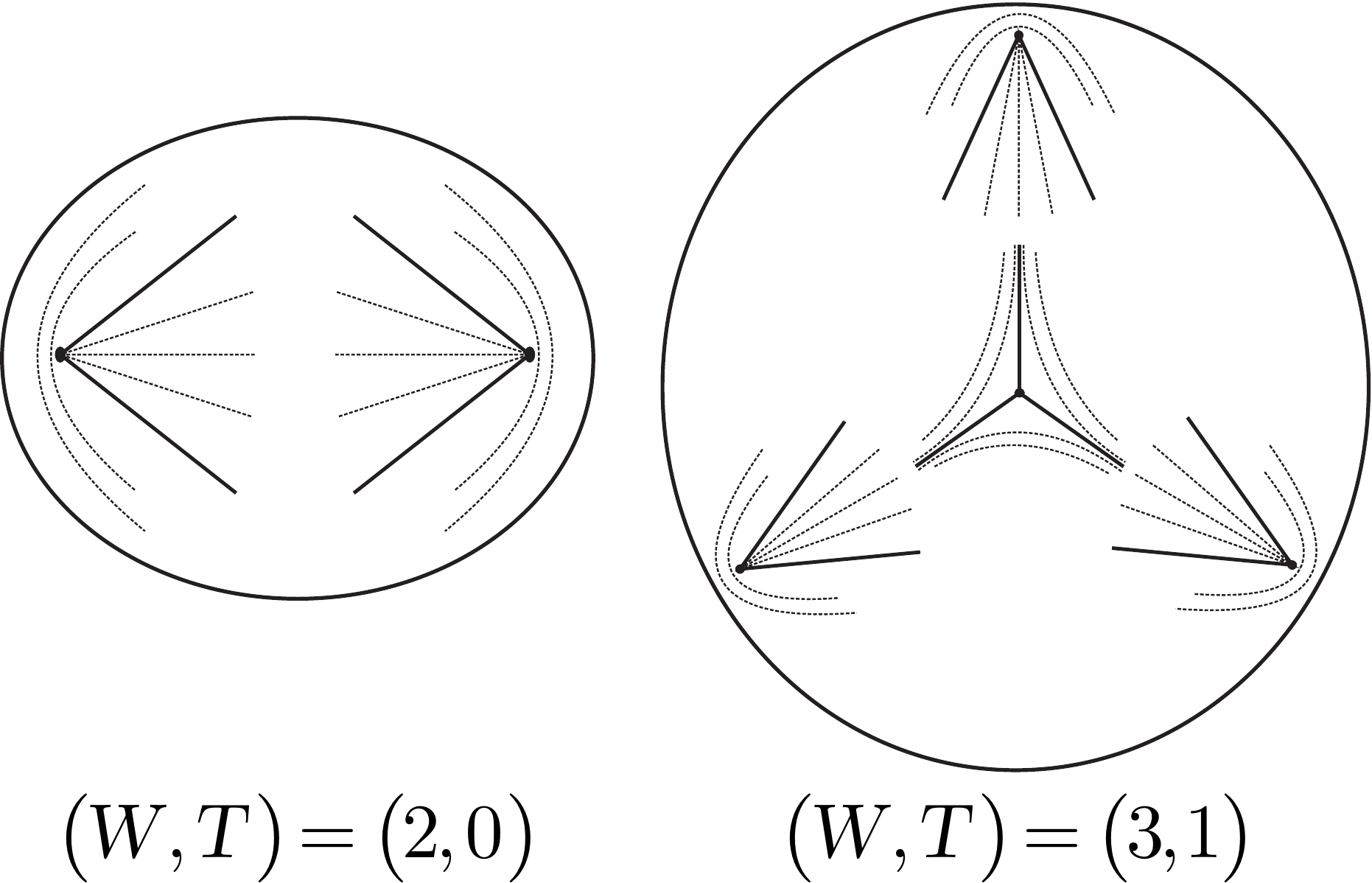}
\caption{Possible topologies inside closed orbits: the $\left(W,T\right)=(2,0)$
configuration (left) and the $(3,1)$ configuration (right). In practice,
we have only observed the simpler $(2,0)$ configuration.}
\label{fig:topo_cycle}
\end{figure}

\section{Application to coherent Lagrangian vortex detection}\label{sec:Application}

Finding closed orbits in planar line fields is the decisive step in
the detection of coherent Lagrangian vortices in a frame-invariant fashion
\cite{Haller2012,Beron-Vera2013,Haller2013a}. Before describing
the algorithmic scheme and showing results on ocean data, we briefly
introduce the necessary background and notation for coherent Lagrangian vortices.

\subsection{Flow map, Cauchy--Green strain tensor and \texorpdfstring{$\lambda$}{lambda}--line field}

We consider an unsteady, smooth, incompressible planar velocity field
$\mathbf{u}(t,\mathbf{x})$ given on a finite time interval $\left[t_{0},t_{0}+T\right]$,
and the corresponding equation of motion for the fluid,
\[
\dot{\mathbf{x}}=\mathbf{u}(t,\mathbf{x}).
\]
We denote the associated flow map by $\mathbf{F}_{t_{0}}^{t_{0}+T}$,
which maps initial values $\mathbf{x}_{0}$ at time $t_{0}$ to their
respective position at time $t_{0}+T$. Recall that the flow map
is as smooth as the velocity field $\mathbf{u}$. Its linearisation can be used to define
the \emph{Cauchy--Green strain tensor field}
\[
\mathbf{C}_{t_{0}}^{t_{0}+T}\coloneqq\left(D\mathbf{F}_{t_{0}}^{t_{0}+T}\right)^{\top}D\mathbf{F}_{t_{0}}^{t_{0}+T},
\]
which is symmetric and positive-definite at each initial value. The eigenvalues
and eigenvectors of $\mathbf{C}_{t_{0}}^{t_{0}+T}$ characterize the
magnitude and directions of maximal and minimal stretching locally
in the flow. We refer to these positive eigenvalues as $\lambda_{1}\leq\lambda_{2}$,
with the associated eigenspaces spanned by the normalized eigenvectors $\xi_{1}$
and $\xi_{2}$.

As argued by \cite{Haller2013a}, the positions of coherent Lagrangian
vortex boundaries at time $t_{0}$ are closed stationary curves of
the averaged tangential strain functional computed from $\mathbf{C}_{t_{0}}^{t_{0}+T}$.
All stationary curves of this functional turn out to be uniformly
stretched by a factor of $\lambda>0$ under the flow map $\mathbf{F}_{t_{0}}^{t_{0}+T}$.
These stationary curves can be computed as closed orbits of the \emph{$\lambda$--line
fields} $\eta_{\lambda}^{\pm}$, spanned by the representing vector
fields
\begin{equation}\label{eq:lambda-line-field}
\eta_{\lambda}^{\pm}\coloneqq\sqrt{\frac{\lambda_{2}-\lambda^{2}}{\lambda_{2}-\lambda_{1}}}\xi_{1}\pm\sqrt{\frac{\lambda^{2}-\lambda_{1}}{\lambda_{2}-\lambda_{1}}}\xi_{2}.
\end{equation}

We refer to orbits of $\eta_{\lambda}^{\pm}$ as \emph{$\lambda$--lines}.
In the special case of $\lambda=1$, the line field $\eta_{1}^{\pm}$
coincides with the \emph{shear line field} defined in \cite{Haller2012},
provided that the fluid velocity field $\mathbf{u}(t,\mathbf{x})$
is incompressible.

We refer to points at which the Cauchy--Green strain tensor is isotropic (i.e., equals a constant multiple of the identity tensor) as \emph{Cauchy--Green
singularities}. For incompressible flows, only $\mathbf{C}_{t_{0}}^{t_{0}+T}=\mathbf{I}$
is possible at Cauchy--Green singularities, implying $\lambda_{1}=\lambda_{2}=1$
at these points. The associated eigenspace fields, $\xi_{1}$ and
$\xi_{2}$, are ill-defined as line fields at Cauchy--Green singularities,
thus generically the line fields $\xi_{1}$, $\xi_{2}$ and $\eta_{1}^{\pm}$
have singularities at these points. Conversely,
the singularities of the line fields $\xi_{1}$, $\xi_{2}$ and $\eta_{1}^{\pm}$
are necessarily Cauchy--Green singularities, as seen from the local
vector field representation in Eq.\ \eqref{eq:lambda-line-field}.

Following \cite{Haller2012,Haller2013a}, we define an \emph{elliptic
Lagrangian Coherent Structure (LCS)} as a structurally stable closed orbit of $\eta_{\lambda}^{\pm}$
for some choice of the $\pm$ sign, and for some value of the parameter
$\lambda$. We then define a \emph{(coherent Lagrangian) vortex boundary}
as the locally outermost elliptic LCS over all choices of $\lambda$.

\subsection{Index theory for \texorpdfstring{$\lambda$}{lambda}--line fields}

In regions where $\lambda_{1}<\lambda^{2}<\lambda_{2}$ is not satisfied,
$\eta_{\lambda}^{\pm}$ is undefined. Such open regions necessarily arise
around Cauchy--Green singularities, and hence $\eta_{\lambda}^{\pm}$
does not admit isolated point-singularities. Consequently, the index
theory presented in \prettyref{sec:Index_lf} does not immediately
apply to the $\lambda$--line field. We show below, however, that
Cauchy--Green singularities are still necessary indicators of closed
orbits of $\eta_{\lambda}^{\pm}$ for arbitrary $\lambda$.

For $\lambda>1$, the set $D_{\lambda}^{2}=\left\{ \lambda_{2}<\lambda^{2}\right\}$,
on which $\eta_{\lambda}^{\pm}$ is undefined, consists of open connected
components. All Cauchy--Green singularities are contained in some
of these $D_{\lambda}^{2}$-components. A priori, however, there may
exist $D_{\lambda}^{2}$-components that do not contain Cauchy--Green
singularities.

On the boundary $\partial D_{\lambda}^{2}$, we have $\lambda_{2}=\lambda^{2}$,
and hence $\eta_{\lambda}^{\pm}$ coincides with $\xi_{2}$ on $\partial D_{\lambda}^{2}$,
as shown in Fig.\ \ref{fig:continuation}. Therefore, we may extend
$\eta_{\lambda}^{\pm}$ into $D_{\lambda}^{2}$ by letting $\eta_{\lambda}^{\pm}(\mathbf{x})\coloneqq\xi_{2}(\mathbf{x})$
for all $\mathbf{x}\in D_{\lambda}^{2}$, thereby obtaining a continuous,
piecewise differentiable line field, whose singularity positions coincide
with those of the $\xi_{2}$-singularities.

\prettyref{thm:topo_line} applies directly to the continuation of
the line field $\eta_{\lambda}^{\pm}$, and enables the detection
of closed orbits lying outside the open set $D_{\lambda}^{2}$. In
the case $\lambda<1$, the line field $\eta_{\lambda}^{\pm}$ can
similarly be extended in a continuous fashion into the interior of
the set $D_{\lambda}^{1}=\left\{ \lambda_{1}>\lambda^{2}\right\} ,$
through the definition $\eta_{\lambda}^{\pm}(\mathbf{x})\coloneqq\xi_{1}(\mathbf{x})$
for all $\mathbf{x}\in D_{\lambda}^{1}$.

\begin{figure}
\centering
\includegraphics[width=0.45\textwidth]{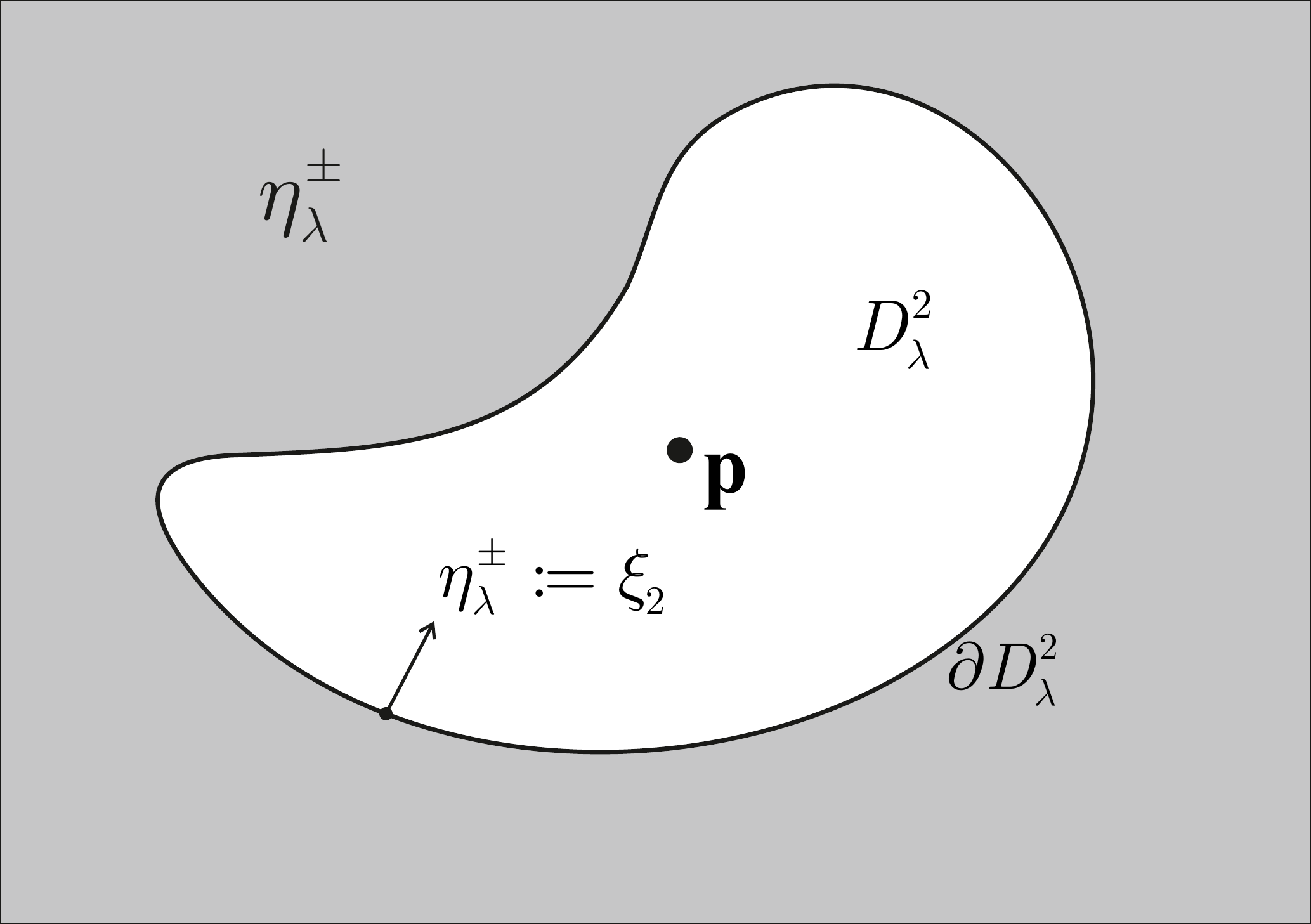}
\caption{The original domain of definition of $\eta_{\lambda}^{\pm}$ (grey), and the domain $D_{\lambda}^{2}$ (white), to which $\eta_{\lambda}^{\pm}$ can be con\-tin\-u\-ously extended via
the line field $\xi_2$. Also shown is a point $\mathbf{p}$ denoting a Cauchy--Green
singularity.}
\label{fig:continuation}
\end{figure}

After its extension into the set $D_{\lambda}=D_{\lambda}^{1}\cup D_{\lambda}^{2}$,
the line field $\eta_{\lambda}^{\pm}$ inherits each Cauchy--Green
singularity either from $\xi_{2}$ or from $\xi_{1}$. A priori, the
same Cauchy--Green singularity may have different topological types
in the $\xi_{1}$ and $\xi_{2}$ line fields. By \cite[Theorem 11]{Delmarcelle1994a},
however, this is not the case: corresponding generic singularities
of $\xi_{2}$ and $\xi_{1}$ share the same index and have the same
number of hyperbolic sectors. Furthermore, the separatrices of the
$\xi_{2}$-singularity are obtained from the separatrices of the $\xi_{1}$-singularity
by reflection with respect to the singularity. In summary, $\xi_{1}$-wedges
correspond exactly to $\xi_{2}$-wedges, and the same holds
for trisectors. For the singularity type classification for $\eta_{\lambda}^{\pm}$,
$\lambda\neq1$, we may therefore pick $\xi_{2}$, irrespective of
the sign of $\lambda-1$.

The singularity-type correspondence extends also to the limit case
$\lambda=1$, i.e., to $\eta_{1}^{\pm}$, as follows. Consider the
one-parameter family of line-field extensions $\eta_{\lambda}^{\pm}$.
By construction, the locations of $\eta_{\lambda}^{\pm}$ point singularities
coincide with those of the $\xi_{2}$-singularities for any $\lambda$.
Variations of $\lambda$ correspond to continuous line-field perturbations,
which leave the types of structurally stable singularities
unchanged. Hence, the types of $\eta_{1}^{\pm}$-singularities must
match the types of corresponding $\eta_{\lambda}^{\pm}$-singularities,
or equivalently of corresponding $\xi_{2}$-singularities. To summarize, we obtain the following conclusion.

\begin{proposition}\label{prop:eta-field}
Any closed orbit of a structurally stable $\eta_\lambda^\pm$ field necessarily encircles Cauchy--Green singularities satisfying Eq.\ \eqref{eq:toporule_line}.
\end{proposition}

\subsection{A simple example: coherent Lagrangian vortex in the double gyre flow}

We consider the left vortex of the double gyre flow \cite{Shadden2005},
defined on the spatial domain $[0,1]\times[0,1]$ by the ODE
\begin{align*}
\dot{x} & =-\pi A\sin(\pi f(x))\cos(\pi y),\\
\dot{y} & =\pi A\cos(\pi f(x))\sin(\pi y)\partial_{x}f(t,x),
\end{align*}
where
\[
f(t,x)=\varepsilon\sin(\omega t)x^{2}+\left(1-2\varepsilon\sin(\omega t)\right)x.
\]
We choose the parameters of the flow model as $A=0.2$, $\varepsilon=0.2$,
$\omega=\pi/5$, $t_{0}=0$, and $T=5\pi/2$.

In the $\lambda$--line field shown in Fig.\ \ref{fig:double_gyre}(a), we
identify a pair of wedge singularities. Any closed $\lambda$--line
must necessarily enclose this pair by \prettyref{prop:eta-field}. This prompts us to define a Poincar\'e
section through the midpoint of the connecting line between the two wedges.
For computational simplicity, we select the Poincar\'e section as horizontal.
Performing a parameter sweep over $\lambda$--values, we obtain the
outermost closed orbit shown in Fig.\ \ref{fig:double_gyre}(a) for a uniform
stretching rate of $\lambda=0.975$. Other non-closing orbits and
the $\lambda$--line field are also shown for illustration. In addition, we
show the line field topology around the wedge pair in the vortex core
in Fig.\ \ref{fig:double_gyre}(b).

In this simple example, the vortex location is known, and hence a
Poincar\'e section could manually be set for closed orbit detection
in the $\lambda$--line fields. In more complex flows, however, the
vortex locations are a priori unknown, making a manual search unfeasible.

\begin{figure}
\centering
\includegraphics[width=.95\textwidth]{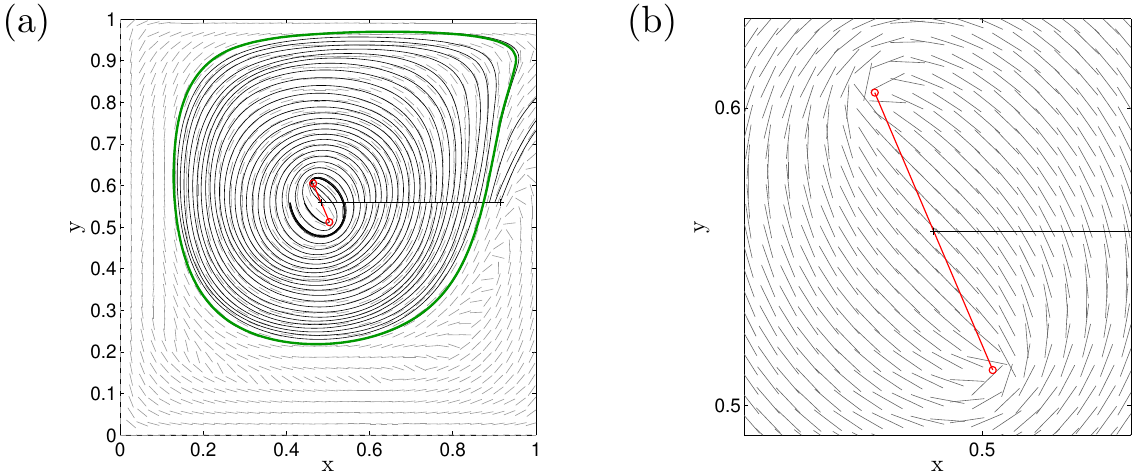}
\caption{(a) Vortex boundary ($\lambda=0.975$) for the left vortex of the
double gyre flow. In the centre, the pair of wedge singularities (red)
determines the topology of the $\lambda$--line field $\eta_{\lambda}^{-}$
(grey) and therefore indicates a candidate region for closed orbits.
The $\lambda$--lines (black) are launched from the Poincar\'e section
(black crosses) to find the outermost closed orbit (green). (b) A
blow-up of the centre of the vortex with the detailed circular topology
of the $\lambda$--line field $\eta_{\lambda}^{-}$ in the vicinity
of the $(2,0)$ wedge pair configuration (cf.\ Fig.\ \ref{fig:topo_cycle}).}
\label{fig:double_gyre}
\end{figure}

\subsection{Implementation for vortex census in large-scale ocean data}\label{sub:implementation}

Our automated Lagrangian vortex-detection scheme relies on \prettyref{prop:eta-field},
identifying candidate regions in which Poincar\'e maps for closed $\lambda$--line
detection should be set up. In several tests on ocean data, we only
found the $(W,T)=(2,0)$ singularity configuration inside closed $\lambda$--lines.
This is consistent with our previous genericity considerations. Consequently,
we focus on finding candidate regions for closed $\lambda$--lines
as regions with isolated pairs of wedges in the $\xi_{2}$ field.
In the following, we describe the procedure for an automated
detection of closed $\lambda$--lines.

\paragraph*{1. Locate singularities.}

Recall that Cauchy--Green singularities are points where\linebreak $\mathbf{C}_{t_{0}}^{t_{0}+T}=\mathbf{I}$.
We find such points at subgrid-resolution as intersections of the
zero level sets of the functions $c_{1}\coloneqq C_{11}-C_{22}$ and $c_{2}\coloneqq C_{12}=C_{21}$, where $C_{ij}$ denote the entries of the Cauchy--Green strain tensor.

\paragraph*{2. Select relevant singularities.}

We focus on generic singularities, which are isolated and are of wedge
or trisector type. We discard tightly clustered groups of singularities,
which indicate non-elliptic behavior in that region. Effectively,
the clustering of singularities prevents
the reliable determination of their singularity type. To this end,
we select a minimum distance threshold between admissible singularities
as $2\Delta x$, where $\Delta x$ denotes the grid size used in the
computation of $\mathbf{C}_{t_{0}}^{t_{0}+T}$. We obtain the distances
between closest neighbours from a Delaunay triangulation procedure.

\paragraph*{3. Determine singularity type.}

Singularities are classified as trisectors or wedges, following
the approach developed in \cite{Farazmand2013a}. Specifically, a
circular neighbourhood of radius $r>0$ is selected around a singularity,
so that no other singularity is contained in this neighbourhood. With a rotating
radius vector $\mathbf{r}$ of length $r$, we compute the absolute
value of the cosine of the angle enclosed by $\mathbf{r}$ and $\xi_{2}$,
i.e., $\cos\left(\angle\left(\mathbf{r},\xi_{2}\right)\right)=\left|\mathbf{r}\cdot\xi_{2}\right|/r$,
with the eigenvector field $\xi_{2}$ interpolated linearly to 1000
positions on the radius $r$ circle around the singularity. The singularity
is classified as a trisector, if $\mathbf{r}$ is orthogonal to $\xi_{2}$
at exactly three points of the circle, and parallel to $\xi_{2}$
at three other points, which mark separatrices of the trisector (Fig.\ \ref{fig:deg_points}).
Singularities not passing this test for trisectors are classified
as wedges. Other approaches to singularity classification can be found
in \cite{Delmarcelle1994} and \cite{Bazen2002}, which we have found
too sensitive for oceanic data sets.

\paragraph*{4. Filter}

We discard wedge points whose closest neighbour is of trisector-type,
because these wedge points cannot be part of an isolated wedge pair.
We further discard single wedges whose distance to the closest wedge
point is larger than the typical mesoscale distance of $2\,^{\circ}\approx200$\,km.
The remaining wedge pairs mark candidate regions for elliptic LCS
(Fig.\ \ref{fig:coherent_eddies}(a)).

\paragraph*{5. Launch $\lambda$--lines from a Poincar\'e-section}

We set up Poincar\'e sections that span from the midpoint of a
wedge pair to a point $1.5\,^{\circ}$ apart in longitudinal direction
(Fig.\ \ref{fig:coherent_eddies}(b)). This choice of length for the
Poincar\'e section captures eddies up to a diameter of $3\,^{\circ}\approx300$\,km,
an upper bound on the accepted size for mesoscale eddies. For a fixed
$\lambda$--value, $\lambda$--lines are launched from 100 initial
positions on the Poincar\'e section, and the return distance $P(x)-x$
is computed. Zero crossings of the return distance function correspond
to closed $\lambda$--lines. The position of zeros is subsequently refined
on the Poincar\'e section through the bisection method. The outermost
zero crossing of the return distance marks the largest closed $\lambda$--line
for the chosen $\lambda$--value. To find the outermost closed $\lambda$--line over all $\lambda$--values,
we vary $\lambda$ from $0.85$ to $1.15$ in $0.01$ steps, and pick
the outermost closed orbit as the Lagrangian eddy boundary. During this process,
we make sure that eddy boundaries so obtained do enclose the two wedge
singularities used in the construction, but no other singularities.

The runtime of our algorithm is dominated by the fifth step, the integration of $\lambda$--lines,
as illustrated in Table \ref{tab:run-time} for the ocean data example in
the next section. This is the reason why our investment in the selection,
classification and filtering of singularities before the actual $\lambda$--line
integration pays off.

\subsection{Coherent Lagrangian vortices in an ocean surface flow}

We now apply the method summarized in steps 1-5 above to two-dimensional
unsteady velocity data obtained from AVISO satellite altimetry measurements.
The domain of the data set is the Agulhas leakage in the Southern
Ocean, represented by large coherent eddies that pinch off from the
Agulhas current of the Indian Ocean.

Under the assumption of a geostrophic flow, the sea surface height
$h$ serves as a streamfunction for the surface velocity field. In
longitude-latitude coordinates $(\varphi,\theta)$, particle trajectories
are then solutions of
\begin{align*}
\dot{\varphi} & =-\frac{\mathrm{g}}{\mathrm{R}^{2}f(\theta)\cos\theta}\partial_{\theta}h(\varphi,\theta,t), & \dot{\theta} & =\frac{\mathrm{g}}{\mathrm{R}^{2}f(\theta)\cos\theta}\partial_{\varphi}h(\varphi,\theta,t),
\end{align*}
where $\mathrm{g}$ is the constant of gravity, $\mathrm{R}$ is the
mean radius of the Earth, and $f(\theta)\coloneqq 2\Omega\sin\theta$
is the Coriolis parameter, with $\Omega$ denoting the Earth's mean
angular velocity. For comparison, we choose the same spatial domain
and time interval as in \cite{Beron-Vera2013,Haller2013a}. The integration
time $T$ is also set to $90$ days.

Fig.\ \ref{fig:coherent_eddies} illustrates the steps of the eddy
detection algorithm. From all singularities of the Cauchy--Green strain
tensor, isolated wedge pairs are extracted (Fig.\ \ref{fig:coherent_eddies}(a))
and closed orbits are found by launching $\lambda$--lines from Poincar\'e
sections anchored at those wedge pairs (Fig.\ \ref{fig:coherent_eddies}(b)).
Altogether, 14 out of the selected wedge pairs are encircled by closed orbits and, hence,
by coherent Lagrangian eddy boundaries (Fig.\ \ref{fig:coherent_eddies}(c)).
The reduction to candidate regions consistent with \prettyref{prop:eta-field}
leads to significant gain in computational speed. This is because
the computationally expensive integration of the $\lambda$--line
field is only carried out in these regions (Table \ref{tab:run-time}).
For comparison, the computational cost on a single Poincar\'e
section is already higher than the cost of identifying the candidate
regions. Note also that two regions contain three wedges, which constitute two
admissible wedge pairs. This explains how 78 wedges constitute 40 wedge pairs altogether.

\begin{figure}
\centering
\includegraphics[height=0.75\textheight]{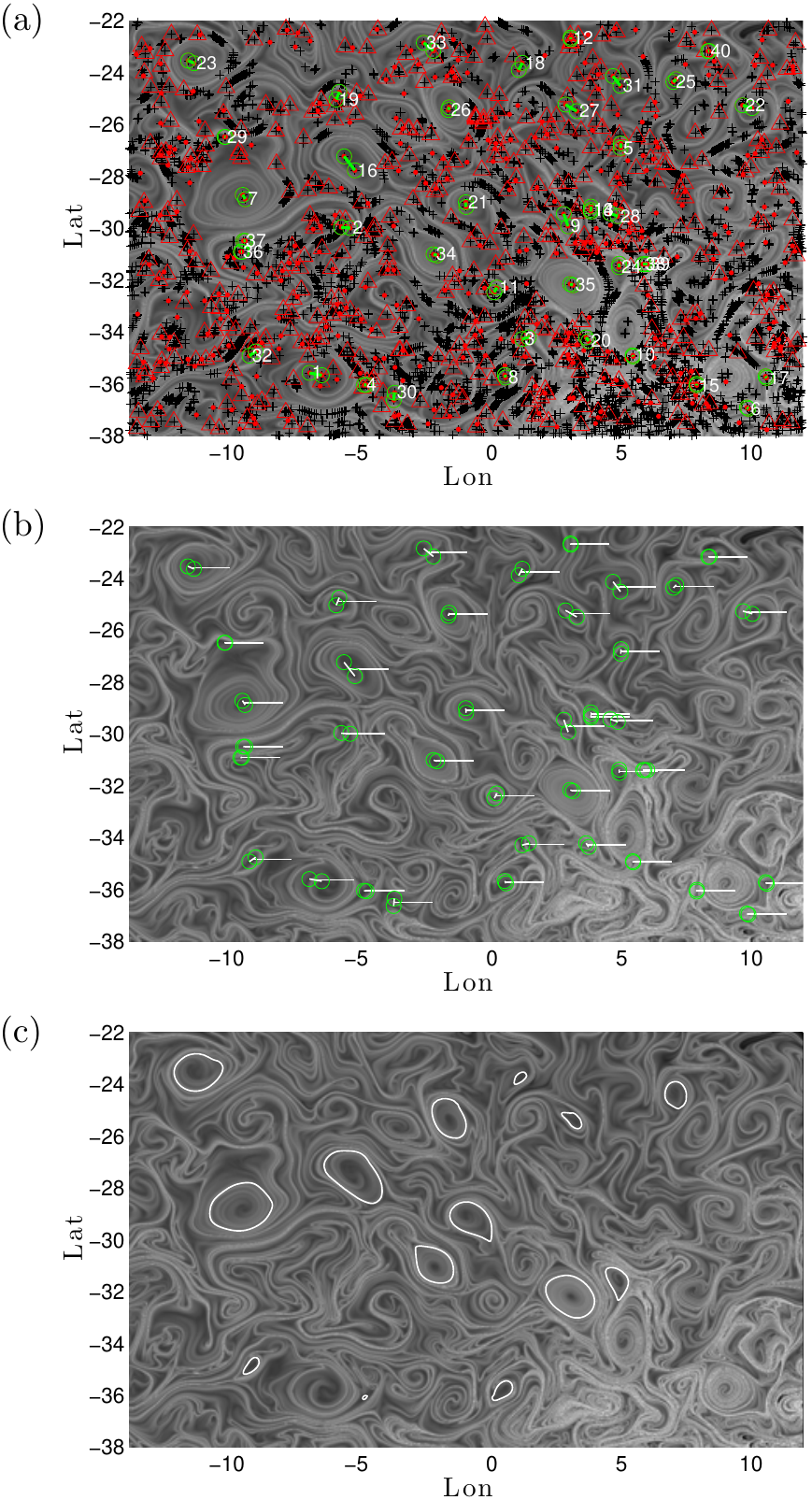}
\caption{Visualization of the eddy detection algorithm for an ocean surface flow. (a) Singularities of
the Cauchy--Green strain tensor of trisector type (red triangles)
and wedge type (green circles: kept, red dots: discarded). Wedge pairs
are candidate cores of coherent eddies. A total of 40 wedge pairs
were finally selected for further analysis out of all singularities
(black crosses) by the procedure described in Section \ref{sub:implementation}.
(b) Poincar\'e sections anchored at the centre of the selected wedge
pairs. Coherent vortex boundaries are found as closed $\lambda$--lines
intersecting these Poincar\'e sections. (c) Boundaries of 14 coherent
eddies on November 24, 2006. The $\log_{10}\lambda_{2}$ field is
shown in the background as an illustration of the stretching distribution
in the flow.}
\label{fig:coherent_eddies}
\end{figure}


\begin{table}
\centering
\begin{tabular}{|l|c|c|}
\hline
 & Runtime & Number of points\tabularnewline
\hline
\hline
1. Localization & $11.0$\,s & 14,211 singularities\tabularnewline
\hline
2. Selection & $12.8$\,s & 912 singularities\tabularnewline
\hline
3. Classification & $85.9$\,s & 414 wedges\tabularnewline
\hline
4. Filtering & $0.5$\,s & 78 wedges\tabularnewline
\hline
5. Integration & $\sim200$\,s / wedge pair / $\lambda$--value & 40 wedge pairs\tabularnewline
\hline
End result & --- & 14 eddies\tabularnewline
\hline
\end{tabular}
\caption{Runtime of the algorithm for the Agulhas data set on a CPU with 2.2\,GHz
and 32\,GB RAM. Since the integration of $\lambda$--lines is the
com\-pu\-ta\-tion\-al\-ly most expensive part, the reduction
of the number of candidate regions to only 40 by application of index
theory yields a significant computational advantage.}
\label{tab:run-time}
\end{table}

\section{Conclusion}\label{sec:Conclusions}

We have discussed the use of index theory in the detection of closed
orbits in planar line fields. Combined with physically motivated filtering
criteria, index-based elliptic LCS detection provides an automated
implementation of the variational results of \cite{Haller2013a}
on coherent Lagrangian vortex boundaries. Our results further enhance
the power of LCS detection algorithms already available in the Matlab
toolbox \textsc{LCS TOOL} (cf.\ \cite{Onu2014}).

Our approach can be extended to three-dimensional flows, where line
fields arise in the computation of intersections of elliptic LCS with
two-dimensional planes \cite{Blazevski2014}. Applied over several
such planes, our approach allows for an automated detection of coherent
Lagrangian eddies in three-dimensional unsteady velocity fields.

Automated detection of Lagrangian coherent vortices should lead to
precise estimates on the volume of water coherently carried by mesoscale
eddies, thereby revealing the contribution of coherent eddy transport
to the total flux of volume, heat and salinity in the ocean.
Related work is in progress.

\section*{Acknowledgment}

The altimeter products used in this work are produced by SSALTO/DUACS
and distributed by AVISO, with support from CNES (\url{http://www.aviso.oceanobs.com}).
We would like to thank Bert Hesselink for providing Ref.\ \cite{Delmarcelle1994a},
Xavier Tricoche for pointing out Refs.\ \cite{Wischgoll2001,Wischgoll2006},
and Ulrich Koschorke and Francisco Beron-Vera for useful comments.

\bibliographystyle{plain}

\end{document}